\documentclass[11pt]{amsart}
\usepackage[english]{babel}
\usepackage{amssymb}
\usepackage{amsthm}
\usepackage{amsmath}
\usepackage{mathtools}
\usepackage{tikz-cd}
\usepackage{graphicx}
\usepackage{subcaption}
\usepackage{caption}

\theoremstyle{definition}
\newtheorem{Def}{Definition}[section]

\theoremstyle{plain}
\newtheorem{Teorema}[Def]{Theorem}

\theoremstyle{definition}
\newtheorem{Ex}[Def]{Example}

\theoremstyle{plain}
\newtheorem{Lemma}[Def]{Lemma}
\newtheorem{Defin}[Def]{Definition}

\theoremstyle{remark}
\newtheorem{rmk}[Def]{Remark}

\theoremstyle{plain}
\newtheorem{Corollary}[Def]{Corollary}

\newcommand{\numberset}{\mathbb}

\newcommand{\R}{\numberset{R}}

\newcommand{\C}{\numberset{C}}
\newcommand{\n}{\nabla}

\newcommand{\de}{\partial}

\newcommand{\proj}{\mathbb{P}}

\title{On a  Modification of the Twistor Space}

\begin{document}

\author{Anna Fino}
\address[Anna Fino]{Dipartimento di Matematica ``G. Peano'', Universit\`{a} degli studi di Torino \\
Via Carlo Alberto 10\\
10123 Torino, Italy\\
\& Department of Mathematics and Statistics, Florida International University\\
Miami, FL 33199, United States}
\email{annamaria.fino@unito.it, afino@fiu.edu}

\author{Gueo Grantcharov}
\address[Gueo Grantcharov]{Department of Mathematics and Statistics \\
Florida International University\\
Miami, FL 33199, United States
\& Institute of Mathematics and
Informatics, Bulgarian Academy of Sciences\\
 8 Acad. Georgi Bonchev Str. 1113 Sofia, Bulgaria}

\email{grantchg@fiu.edu}

\author{Alberto Pipitone Federico}
\address[Alberto Pipitone Federico]{Department of Mathematics \\
University of Tor Vergata\\
Via delle Ricerca Scientifica 1 \\
 00133 Roma, Italy}
\email{pipitone@mat.uniroma2.it}



\maketitle

\begin{abstract} In the paper we  construct  a modification $S(M)$ of the twistor space of a K\"ahler scalar flat surface $M$ and study its complex-geometric and metric properties.  In particular, we construct complete  balanced metrics on $S(M)$ and show that 
$S(M)$ can not be  K\"ahler when $M$  is a compact  simple  hyperk\"ahler manifold.
 \end{abstract}

\section{Introduction}

\smallskip

Twistor theory was initiated by R. Penrose  \cite{Penrose} in the 60's as a tool to transform solutions of some physical nonlinear differential equations into holomorphic objects, which could be studied with the tools of complex geometry. The theory was formalized and generalized by Atiyah,  Hitchin and  Singer \cite{AHS}. In \cite{AHS} the twistor space ${\mbox{Tw}}(M)$ of a 4-dimensional oriented Riemannian manifold $(M, g)$  is defined as the associated bundle to the principal bundle of orhonormal positive frames, with fiber all complex structures in ${\mathbb R}^4$ compatible with the Euclidean scalar product. It is naturaly equipped with a tautological almost complex structure $J$, defined in terms of the projection $\pi: {\mbox {Tw}} (M) \rightarrow M$. More precisely, the Levi-Civita connection induces a splitting $$T({\mbox {Tw}}(M)) = {\mathcal H} \oplus {\mathcal V},$$  where ${\mathcal H}$ is the horizontal space such that $\pi_*: {\mathcal H} \rightarrow TM$ is an isomorphism and $\mathcal V$ is tangent to the fibers of $\pi$. The fibers of $\pi$ are naturally identified with $S^2\equiv {\mathbb {CP}}^1$ and $J$ on ${\mathcal V}$ is defined as the canonical complex structure on $S^2$. While on $\mathcal H$ at a point $p\in {\mbox {Tw}} (M)$ considered as a complex structure on $T_{\pi(p)}M$ the structure $J|_p$ is just the pull-back of $p$ to ${\mathcal H}$ via the isomorphism $\pi_*$.

We recall that  for a $4$-dimensional   oriented Riemannian  manifold $(M,g)$,  the space of $2$-forms 
decomposes as self-dual and anti-self-dual $2$-forms 
$$
\Lambda^2 = \Lambda^2_+ \oplus \Lambda^2_-,
$$ 
according to the eigenvalues of the Hodge star operator $\star$.  A  $2$-form $\alpha$  is called self-dual if $\star \alpha = \alpha$. Then the curvature operator of $(M, g)$ takes the following  form according to the above   decomposition of 2-forms
$$
{\mathcal R} = \left( \begin{array}{cc} W_+ +  \frac {\text{Scal} (g)}{12}  {\mbox {Id}} & \text{Ric}^0 \\[4 pt]  \text{Ric}^0  & W_- +  \frac {\text{Scal}(g)}{12}  {\mbox {Id}}  \end{array} \right ),
$$
where  $W$ is the Weyl tensor, $\text{Ric}^0$ comes from  the trace-free Ricci curvature and $\text{Scal} (g)$ denotes  the scalar curvature of $g$. If $W_+ =0$,  then $g$ is said to be an anti-self-dual metric.

 The main result of \cite {AHS},  which is also implicit in the works of R. Penrose, is that $({\mbox{Tw}}(M),J)$ is a complex manifold, i.e. $J$ is integrable, if and only if $(M,g)$ is anti-self-dual. The identification of the fibers of $\pi$ can be realized as the unit spheres in the space $\Lambda^2_+$ of the self-dual $2$-forms.

Since \cite{AHS} the twistor theory has been generalized in many ways. A straightforward generalization to higher dimensional manifolds leads to complex twistor space only for conformally flat manifolds \cite{Berard-B-Ochiai, O'Brian-Rawnsley}. J. Eells and S. Salamon  \cite{ES} considered the complex structure $\tilde{J}$ which on $\mathcal V$ is with the opposite sign and applied the construction to the minimal surfaces theory. It is known however that the induced almost complex structure on the twistor space is never integrable, but it is very useful in harmonic maps theory.  S. Salamon extended the theory to higher dimensional quaternionic manifolds  \cite{Salamon} which also have  an  integrable almost complex structure. The twistor space of ${\mathbb R}^{4n}$ can be identified with the total space of the bundle $\mathcal{O}(1)^{2n}$  over ${\mathbb {CP}}^1$. This construction  was generalized by C. Simpson to the so-called mixed twistor structures \cite{Simpson}, in relation with  mixed Hodge theory.

More recently, the  twistor space approach provided also a  powerful tool for understanding the  geometry of Joyce structures on complex manifolds \cite{Joyce, Bridgeland}. A Joyce structure involves a one-parameter family of flat and symplectic non-linear connections on the tangent
bundle $TM$  and gives rise to a  complex  hyperk\"ahler structure on $TM$.

For any hypercomplex manifold the twistor space admits also a holomorphic projection onto ${\mathbb {CP}}^1$, generalizing the projection $\pi_1: \mathcal{O}(1)^{2n}\rightarrow  {\mathbb {CP}}^1$. The approach in \cite{Bridgeland} uses this property for the twistor space of $TM$, for a manifold $M$ endowed with a Joyce structure.  In another generalization,  C. LeBrun \cite{Lebrun} considered a ramified cover $\mu:\Sigma\rightarrow {\mathbb {CP}}^1$ which can be lifted to a ramified cover of the twistor space. For appropriate $\Sigma$ the construction leads to a compact complex manifold with trivial canonical bundle and was used by T. Fei \cite{Fei} to find solutions of the Hull-Strominger system.  Other works study  the twistor spaces of generalized complex structures \cite{DM, Deschamps} and maps between twistor spaces \cite{Davidov}.

In this paper, we consider a different modification of the twistor spaces of K\"ahler scalar flat surfaces $(M, I, g, \omega)$. For such surfaces the structure group of the bundle $\Lambda^2_+ M$ of self-dual $2$-forms  can be reduced to $U(1)$ and any surface $S$ of rotation in ${\mathbb R}^3$ could be the fiber of an associated bundle $S(M)$. Taking a map from the unit sphere to  the surface $S$ that  commutes with the rotations  defines an induced almost complex structure $J$ on $S(M)$. We first investigate the integrability of such structure. Our main result (Theorem \ref{Th2.5})  shows that  $J$   is integrable on $S(M)$ if and only if the map induces a biholomorphism between $S(M)$ and a domain in ${\mbox{Tw}}(M)$.  We  then consider in more details the case where $S(M)$ is ${\mbox{Tw}}(M)-\{D_1,D_2\}$, namely the twistor space with the sections corresponding to $\pm I$ deleted. We note that the images of these sections $D_1, D_2$,  represent  the canonical divisor of ${\mbox{Tw(M)}}$ as $K=2D_1+2D_2$ (cf. \cite{Hitchin, Pontecorvo}).

The metric properties of the twistor spaces were also an object of intensive study.  N. Hitchin \cite{Hitchin}  showed that ${\mbox{Tw}}(M)$ is K\"ahler if and only if $(M, g)$ is isometric to $S^4$ or ${\mathbb {CP}}^2$. In general, when  a 4-dimensional oriented Riemannian manifold $(M, g)$  is anti-self-dual, the natural metric on ${\mbox {Tw}}(M)$ is balanced \cite{Michelsohn, Gauduchon2, Oleg}.  In particular
the twistor space of any compact anti–self–dual $4$-manifold admits a balanced metric, but only the ones which are in addition Einstein with positive scalar
curvature admit a K\"ahler metric. Further works about the almost Hermitian geometry of the twistor spaces include  \cite{ADM, Deschamps2, Catino-Dameno-Mastrolia}.

 In  Section  3 we  show that any conformal change of the metric on the fiber of ${\mbox{Tw}}(M)$ still provides a balanced metric  on the complex manifold $(S(M), J)$ (Theorem \ref{balancedness}), and some of these balanced metrics  are complete (Corollary \ref{Cor-completness}).  In Section 4 we prove  that, if $Tw(M)$ is the twistor space of a  compact simple  hyperk\"ahler manifold $M$,   as long as one deletes from $Tw(M)$
finetely many divisors,  $Tw(M)$ remains non-K\"ahler (Theorem \ref{Th4.1}).  Therefore,  $S(M)$ cannot be    K\"ahler  when $M$ is a compact simple hyperk\"ahler manifold. We note however, that the twistor space of ${\mathbb R}^{4n}$ becomes Stein when one divisor corresponding to a section is deleted.

The paper is organized as follows. In Section 2  we introduce the construction of the modified twistor space $S(M)$,  we define the induced almost complex structure $J$, and we prove the integrability criterion. 
Section 3 is devoted to the construction of balanced metrics on $S(M)$.  In Section 4 we study the modified twistor space of a hyperk\"ahler manifold.

\smallskip

\textbf{Acknowledgments:}  The authors thank Misha Verbisky for the useful discussions and for sharing with us the main idea of the proof of Theorem 4.1. Anna Fino  is  partially supported by Project PRIN 2022 \lq \lq Real and complex manifolds:
Geometry and Holomorphic Dynamics”, by GNSAGA (Indam) and by   a grant from the
Simons Foundation (\#944448).  Gueo Grantcharov is partially supported by a grant from
the Simons Foundation (\#853269).  Some of the work was done while the third author visited the  Institute of Mathematics and Informatics of the Bulgarian Academy of Sciences,  he is grateful for their hospitality.

\section{Construction and basic properties of the modified twistor space $S(M)$}

\subsection{Construction of $S(M)$} 

Let $(M,g)$ be a four dimensional Riemannian manifold and let $\Lambda^2_+ M \to M$ be the sub-bundle of the bundle of 2-forms $\Lambda^2M$, consisting of the self-dual  $2$-forms, that is,   the $2$-forms  with eigenvalue $1$ with respect to the Hodge star  operator. The metric $g$ clearly induces a bundle isomorphism $\Lambda^2 TM \to  \Lambda^2 M$  and in the computations we will use the latter. \\
The fiber of the projection $\Lambda^2_+ M \to M$  is a real three dimensional vector space. Consider now the subbundle $\pi: Tw(M) \to M$ given by the $2$-forms  with norm $1$  with respect to  the induced metric by $g$. The structure group of $\pi$  lies in $SO(3)$, hence this fiber bundle is well defined. The fiber of this new bundle is clearly diffeomorphic to $S^2$ and $Tw(M)$  can be identified through $g$ with  the bundle of $g$-compatible almost complex structures which are also orientation preserving. More precisely, 
 for any section of  $\Lambda^2 TM$,   $\sigma= \sum_{ij} a_{ij} X_i \wedge X_j$, with $X_i$  vector fields,  we have that the induced metric $g$    is defined, by the  bilinearity, from
\[
g(X_i\wedge X_j, Y_i \wedge Y_j)=\frac{1}{2}\det(g(X_i,Y_j)).
\]
The operator $K$ which identifies the points of $Tw(M) = S(\Lambda^+)$ with almost complex structures is defined
\[
g(K_a X,Y)=2g(a, X\wedge Y)
\]
for any $a\in S(\Lambda^2_+M)$.  We also denote by $g$ the metric induced on  the dual space $\Lambda^2 M$ of $2$-forms.
The Levi-Civita connection of $g$ on  $\Lambda^2 TM \to M$, which comes from the Levi-Civita on $TM \to M$, has a curvature $R$, which is related to the Riemannian curvature of $TM$ via
\[
R(X,Y) (Z \wedge T)= R(X, Y) Z \wedge T + Z \wedge R(X, Y)  T,
\]
where $R$   denotes also the Riemannian curvature of $TM$.
We define a new operator $\rho \in  \Lambda^2 M \otimes \text{End}(\Lambda^2 TM)$ given by
\[
\rho(X \wedge Y)(Z \wedge T)=R(X,Y)(Z \wedge T).
\]
Since $M$ has real dimension four, one has a self-adjoint real endomorphism  $\mathcal{R}$  of $\Lambda^2 M$ defined by
\[
g(\mathcal{R}(X\wedge Y), Z \wedge T)=g(R(X, Y) Z, T).
\]
In the following we will let $\mathcal{R}$ act on an element of the form $v\times w$ with $v,w\in \Lambda^2_+$,  then
\[
g(\mathcal{R}(v\times w), \xi ) =   g  (\rho ( \xi  )v,  w    )
\]
for any $\xi \in\Lambda^2$.

We  recall  that a Hermitian surface  $(M, I, g,\omega)$, where $\omega(X,Y)=g(IX,Y)$,  is K\"ahler if and only if $\n \omega=0$, where $\n$ is the Levi-Civita connection   of $g$. Since such an $\omega$ is always a $2$-form in  $\Lambda^2_+ M$ with norm $1$, the K\"ahler condition provides two parallel sections of $\pi$, interchanged by a $-1$. This provides two different local trivializations $(\Lambda^2_+ M)|_U \cong  U \times S^2$, since the two sections might be viewed as maps $U  \to \{  (0,0,\pm 1)  \} \subset S^2$. This also shows that the holonomy action of $SO(3)$ on $\Lambda^2_+ M$ is reduced to $U(1)$-action and that the same holds for the structure group of the fiber bundle $\pi$. \\
This allows one to define a subbundle $S(M)$ of $\Lambda^2_+ M$ in the following way.

\begin{Defin} \label{DedS(M)}  Let  $(M, I, g, \omega)$  be a K\"ahler surface.  Fix a point $p\in M$  and let $S$ be a surface embedded in $(\Lambda^2_+ M)_p$, invariant under the action of $U(1)$.  $S(M)$ is the  subbundle  of  $\Lambda^2_+ M$, defined by $S$ and  with structure group  reduced to $U(1)$.
\end{Defin}

Notice   that  $S(M)$ does not intersect the image of the section of $\pi$ corresponding to  the K\"ahler form.  
Moreover, it is clear that a bundle map $F:S(M) \to Tw(M)$ is equivalent to a $U(1)$-equivariant map $f: S\to S^2$. \\
Given such a map, it is possible to define an almost complex structure $J$ on $S(M)$ in the following way. \\
Let $(e_1, e_2, e_3, e_4)$ be a local orthonormal frame for  a K\"ahler surface  $(M,  I,  g, \omega)$, so that $(s_1, s_2, s_3)$ is an orthonormal frame for the natural metric on $\Lambda_+^2 M $ and $s_1 = \omega$ is parallel.
Let $\epsilon$ be the Gauss map sending $p\in S$ to the normal vector to $S$ in $p$. It is clearly equivariant and we have the following identification $T_p S=\epsilon(p)^{\perp}$.
\begin{Lemma} 
For an oriented Riemannian 4-manifold $M$,  the horizontal space $\mathcal{H}$ of the induced connection on $\Lambda_+^2 M $ is tangent to $Tw(M)$.    If,  in addition,  $(M,g, I, \omega) $ is K\"ahler, the horizontal space $\mathcal{H}$ is also  tangent to $S(M) $.
\end{Lemma}
\begin{proof}
	The connection on $\Lambda_+^2 M $ is defined by $$\n(v\wedge w)= \n(v)\wedge w+ v \wedge \n (w).$$ For a general vector bundle $\pi: E \to M$ with a connection $\n$, the horizontal space $H_e \subset T_eE $ is spanned by $s_*(T_{\pi(e)} M) $ where $s$ is a local section satisfying $s(\pi(e))=e$ and $(\n s)|_{\pi(e)} =0$. \\
	 Fix $p\in Tw(M) $, and consider an orthonormal frame $(s_1, s_2, s_3)$  for the natural metric on $\Lambda_+^2 M $,  then $s=a^i(x) s_i (x)$  satisfies $(\n_X s)|_{\pi(e)} =0$ if and only if $d a^i (X)s_i=- a^i (\n_X s_i) $. If this is the case, 
	\[
	s_*(X)=X^k \de_k +(\de_j  a^i) X^j s_i = X^k \de_k - a^i (\n_X s_i).
	\]
	Then $s_* (X) $ is tangent to $S^2$, as $g(\n_X s_i, s_i)=  \sum_{i=1}^4 \omega ^i_i (X)=0   $, with $(\omega ^i_j)$ the connection forms for $\n$ on $M$ with respect to $(s_i)$. \\
	Assume now  that  $(M, I, g, \omega)$ is K\"ahler and that $s_1 = \omega$,  so that $\n s_1 =0 $. By  the compatibility of the Levi-Civita connection  $\n$ with the metric, one finds that 
	\[
	\n_X s_2 =\beta(X) s_3 \qquad \n_X s_3 =-\beta(X) s_2,  
	\]
	with $\beta \in \Lambda^1(M)$. Fix a point $p$ in any rotational surface $S$, then 
	\[
	\text{proj}_{II}\big( s_*(X)\big)=-a^3(x) \beta(X) s_2 + a^2(x) \beta(X) s_3
	\]
	which is clearly still tangent to $S$ at $p$ ($\text{proj}_{II}$ denotes the projection on the fiber).
	\end{proof}
 This allows one to define an almost complex structure $J$  on $S(M)$ as follows. Fix $p\in S(M) $ and $\xi \in T_p S(M)$. The restriction of $\n$ to $S(M)$ induces a splitting $\xi =\xi^h +\xi^v $ and we define
 	\begin{equation} \label{formula(1)}
 	J_p(\xi^h)=(K_{F(p)} (\pi_*\xi) )^h_p  \qquad  J_p(\xi^v)=\epsilon(p) \times \xi^v
 	\end{equation}
 	where $K_{F(p)}$ is the $g$-compatible  almost complex structure  of $T _{\pi(p)}M $ corresponding to $F(p)$, $(\cdot)^h$ denotes  the horizontal part, $(\cdot)^h_p$ denotes  the horizontal lift to $p$, and $\times$ is the standard cross product in $\R^3$. \\
  Let $h$ be the Riemannian metric on $S(M)$ given by $$h=\pi^* g + g^v, $$ with $g^v$ the restriction of $g$ on $\Lambda^2_+ M $. Then we have
  \begin{Lemma}
      The almost complex structure $J$ on $S(M)$ is  compatible with the Riemannian metric $h$.
  \end{Lemma}
  \begin{proof}
      Follows from the fact that the horizontal and vertical spaces are orthogonal and by   \eqref{formula(1)}.
  \end{proof}

  \smallskip

To prove the integrability of $J$, we use the fact that the Nijenhuis tensor $N_J$ can be written as
$$
\begin{array}{lcl}
h( N_J(X,Y),Z ) &=  & (D_X \Omega) (JY,Z) -  (D_{JY}\Omega) (X, Z) \\[3pt]
&& - (D_{Y}\Omega) (JX, Z) + (D_{JX}\Omega) (Y, Z),
\end{array}
$$
where $\Omega(\cdot ,\cdot)=h( \cdot ,J \cdot)$ and $D$ is the Levi-Civita connection of $h$ on $S(M)$. \\

\begin{Lemma}
    Let $(M,g)$ be an oriented Riemmanian 4-manifold and let  $p\in S(M)$, $V$  be vertical in $p$, and $X, Y\in T_{\pi(p)}M$. The following formulas hold.
 \begin{enumerate}
 	\item $g(p\times V, K_{p} X \wedge  Y         ) = g(p\times V, X \wedge K_{p} Y         ) = g( V, X \wedge  Y         ),   $ 
 	\item $\mathcal{V}(D_{X^h} Y^h)_{p}= \frac{1}{2} \rho_{\pi(p)}(X \wedge Y) p, $
 	\item $(D_{V} X^h)_{p}= \frac{1}{2} (\rho_{\pi(p)}(p \times V)X)^h_p, $
 	\item $g(\epsilon(p) \times \rho(X \wedge Y) p, U)=-g(\mathcal{R}(p \times \epsilon(p) \times U), X\wedge Y).    $
 \end{enumerate}
\end{Lemma}
Then the computations of the Nijenhuis tensor of  the complex structure $J$ on $S(M)$ are the following.
 \begin{Lemma}
 Let $X,Z$ be vector fields on $M$ and $U$ a vertical vector field on $S(M)$. Then $$ h(N(X^h,U),Z^h) = 2g(  J f_*(U)      - f_*(JU), X \wedge Z)$$ and $h(N(X^h,U),Z^h)=0$  if and only if $f: S\to S^2$  is holomorphic. 
 \end{Lemma}
 
\begin{proof}

We start with:
 	\[
 	\begin{split} 
 	 h(N(X^h,U),Z^h)=&(D_{X^h} \Omega)(JU, Z^h)-(D_{JU} \Omega)(X^h, Z^h) -(D_{U} \Omega)(JX^h, Z^h) \\
 	 & + (D_{JX^h} \Omega)(U, Z^h) \\
 	  = &\frac{1}{2}\bigg(  g(\mathcal{R} (p\times U),X\wedge Z)-   g(\mathcal{R}(p \times J U), X \wedge J_{f(p)}Z) \bigg)    \\ 	
 	  &+  \frac{1}{2}g( \mathcal{R}(p \times J U), X \wedge J_{f(p)}Z + J_{f(p)}X \wedge Z)-2g(f_*(JU), X \wedge Z) \\
 	  &+  \frac{1}{2}g( \mathcal{R}(p \times  U), J_{f(p)} X \wedge J_{f(p)}Z -X \wedge Z)-2g(f_*(U), J_{f(p)}X \wedge Z) \\
 	  &+\frac{1}{2}\bigg(  -g(\mathcal{R} (p\times JU), J_{f(p)}X \wedge Z   ) -    g(\mathcal{R}(p \times U), J_{f(p)}X \wedge J_{f(p)}Z) \bigg)    \\ 	
 	  = & -2g(f_*(JU), X \wedge Z) -2g(f_*(U), J_{f(p)}X \wedge Z) \\
 	  = &  -2g(f_*(JU), X \wedge Z) - 2g(-f(p) \times f_*(U), X \wedge Z)  	\\
    = & 2g(  J f_*(U)      - f_*(JU), X \wedge Z).
  \end{split}
	\]
	Then note that the horizontal part of $N(X^h,Z^h)$ vanishes, since 
\begin{equation}\label{eq2}
(D_{X^h} \Omega)(Y^h, Z^h)= 2g(\mathcal{V} f_*(X^h), Y \wedge Z).
\end{equation}
Some easy computations show that 
$$
\begin{array}{lcc}
 h(N(X^h,Y^h),U) &= &g( p \times U, \mathcal{R}(JX \wedge JY- X \wedge Y))\\[3pt]
 &&+ g( p \times JU, \mathcal{R}(X \wedge JY+J X \wedge Y)).
\end{array}
$$
Since $\n s_1=0$, we have  $\mathcal{R}(s_2)=\mathcal{R}(s_3)=0$, by  the previous identities for $\n s_2$ and $\n s_3$ and  the Bianchi identities. Moreover, $g(\mathcal{R}(s_1), s_1)=\frac{s}{3}  $, with $s$ the scalar curvature of $M$. Thus, being scalar flat is clearly sufficient for the integrability of $J$. It is also necessary: substitute ciclically $s_i$, $i=1,2,3$,  into the formula above. \\ \\
\end{proof}
\begin{rmk} 
Note that,  since  $f: S\to S^2$  must be $U(1)$-equivariant and $S$ is a surface of rotation, if $S$ is compact then it has genus $0$. This is the case considered in \cite{ADM}.
\end{rmk}

As a consequence we have:
\begin{Teorema} \label{Th2.5}
Let $(M, g)$ be a K\"ahler scalar-flat  $4$-manifold. The almost complex structure $J$ on $S(M)$ is integrable if and only if the map $f$ is holomorphic. In this case $F$ is a biholomorphism from $S(M)$ to a domain in $Tw(M)$.
\end{Teorema}

\subsection{Necessary and sufficient conditions for the existence of $f$}

\smallskip

By the equivariance condition, $f:S \to S^2$ is determined by its restriction to a meridian of $S$, say $\{x=0\} \cap S $. Recall that, for a smooth manifold $M$ with $\dim_{\R}(M)=2$,  the choice of a complex structure on $M$ is equivalent to the choice of a conformal class. A smooth, non constant map between such manifolds, $(M,I)\to (N,J)$, is holomorphic if and only if  it is conformal and preserves the orientation. \\
Let
\[
(\rho_S(z)\cos(\theta), \rho_S(z)\sin(\theta), z)  \qquad 
(\rho_{S^2}(\zeta)\cos(\alpha), \rho_{S^2}(\zeta)\sin(\alpha), \zeta)
\]
be two parametrizations for respectively $S$ and $S^2$. Here $\rho_S$ is the map which associates to a point $p\in S$ the distance $d(p, \{x=y=0\})$. For example $\rho_{S^2}(\zeta)=(1-\zeta^2)^{1/2}$ when it is defined. Clearly $\rho_S$ determines $S$. \\
Each couple $(\frac{\de}{\de \theta}, \frac{\de}{\de z}) $ and $(\frac{\de}{\de \alpha}, \frac{\de}{\de \zeta}) $ is orthogonal and the equivariance of $f$ 
implies that in these coordinates
\[
f(\theta, z)=(\theta, \phi(z)).
\]
In particular, the differential of $f$ is diagonal, which imposes  a constraint on $\phi'(z)$ for 
$f$  to be holomorphic. From  Theorem \ref{Th2.5}, we obtain: \\

\begin{Corollary}
In the above notation, $f$ is holomorphic if and only if 
\[
\phi(z) = \pm\sqrt{1 -  e^c  ( \rho_S (z))^{- 2}}
\]
for some constant $c$.
\end{Corollary}

\begin{proof}
We have that
\[
\left\| \frac{\de}{\de \theta} \right\| = \rho_{S}(z)       \qquad
\left\| f_*\frac{\de}{\de \theta} \right\|= \left\| \frac{\de}{\de \alpha} \right\|= \rho_{S^2}(\phi(z))
\]
and
\[
\left\| \frac{\de}{\de z} \right\|=|\rho_{S}'(z)| \qquad 
\left\| f_*\frac{\de}{\de z} \right\|= \left\| \phi'(z) \frac{\de}{\de \zeta} \right\|= |\phi'(z)| |(\rho_{S^2}'|_{\zeta=\phi(z)})|.
\]
This implies that
\[
\frac{\rho_{S}(z)}{\rho_{S^2}(\phi(z))}= \frac{|\rho_{S}'(z)|}{|\phi'(z)| |\rho_{S^2}'(\phi(z))|}
\]
or more explicitly
\[
|\phi'(z)|=    \frac{|\rho_{S}'(z)|    }{ \rho_{S}(z)}
                \frac{|(1-\phi(z)^2)|}{|\phi(z)|}.
\]
We need to solve for $\phi$, given $\rho_S$. \\
Up to a sign we have
$$
    \frac{\phi(z)} { (1-\phi(z)^2)}  \phi'(z) =   \frac{\rho_{S}'(z)    }{ \rho_{S}(z).}
              $$
              So
              $$
             [ \log ((1-\phi(z)^2) ]' =  - 2  [\log  (\rho_{S}  (z))]' 
              $$
              from which we get
              \[
              \log  \left ( (1-\phi(z)^2) (\rho_S (z))^2 \right) =c, 
              \]
              with c a constant and so 
              \[
              \phi(z)^2 = 1 -  e^c  ( \rho_S (z))^{- 2}.
              \]
\end{proof}

\smallskip

\begin{rmk}\label{Rmk1}
    In the above case the modified twistor $S(M)$ space is biholomorphic to $Tw(M)\setminus \{D_1, D_2\}$, where $D_1, D_2$ are the two divisors given respectively by the images of the K\"{a}hler form $\omega$ and  $-\omega$. Indeed, $f$ is clearly fiberwise either constant or bijective. In the first case, however, for any $m\in M$, we would have $$f|_{p^{-1}(m)}\equiv \pm  \, \omega_m,$$ which we exclude (otherwise the complex structure on $S(M)$ would have no interesting properties). \\
\end{rmk}

\section{Properties of $Tw(M)$ and $S(M)$ for a K\"ahler scalar-flat surface $M$}
We begin with some known facts about the geometry of twistor spaces and especially the ones of the K\"ahler scalar flat surfaces. 
It is known that all twistor spaces of compact  anti-self-dual 4-manifolds admit a balanced metric, but only the ones which are in addition Einstein with positive scalar curvature admit a K\"ahler metric. We will consider the construction of balanced metrics more explicitly and relate them to the metrics on $S(M)$. We first recall that the splitting $\Lambda^2M=\Lambda^2_+M+\Lambda^2_-M$ clearly induces a splitting $H^2(M, \R)=H^2_+(M, \R)\oplus H^2_-(M, \R)$. Then one has, using Serre duality, the following result.
\begin{Teorema} $($\cite[Theorem 5.3.]{Eastwood-Singer}$)$ The twistor space  $Tw(M)$ of a scalar-flat K\"ahler surface $M$ has the Fr\"ohliher spectral sequence degenerating at $E_1$ level if and only if it does not admit any non-parallel holomorphic vector fields. In this case, $$h^{2,2}(Tw(M)) = h^{1,1}(Tw(M)) = 1+b_-(M).$$
 \end{Teorema}
Using this fact, we can show that when $b_-(M)=1$ the set of Dolbeault $(2,2)$-classes admitting a positive representative is an open cone in $H^{2,2} (Tw(M))$.

\hfill

We will now present a method to construct many balanced metrics on the complex manifolds $Tw(M)$ and $(S(M), J)$,  where  $(M, I, g, \omega)$ is a scalar-flat K\"ahler  surface.  By Remark \ref{Rmk1} above,  this is equivalent to produce balanced metrics on $Tw(M)\setminus \{D_1, D_2\}$. It is well known that the standard metric  with associated fundamental form   $$\Omega|_p=  \pi^*_p  \, \omega_p  +\omega_{FS}$$  where $\pi: Tw (M) \rightarrow M$ and  $\omega_p (X,Y) = g(I_pX,Y)$, is balanced on $Tw(M)$, though not complete on $Tw(M)\setminus \{D_1, D_2\}$ (\cite{Michelsohn, Oleg}).
The proof is based on the fact that  the $4$-form $\Omega^2$ is a sum of 3 terms. One term vanishes for dimensional reason, the other is just a pull-back of the volume form on $M$ , so the only non-trivial point to prove is $(d \pi^*\omega_p) \wedge \omega_{FS} = 0$. This follows from the fact that $d\pi^*\omega_p(X^h, Y^h, Z^h) = 0$,  for any horizontal vectors $X^h,Y^h,Z^h$, so $d\pi^*\omega_p = \omega_p\wedge \alpha$,  where $\alpha$ is vertical. But then $\alpha\wedge\omega_{FS} = 0$ for dimensional reasons.
 Now we introduce a partial conformal change, i.e. consider $${\Omega_{h}|}_p= \pi^*_p\omega_p  +e^h\omega_{FS},$$
where $h$ is a  smooth rotationally  invariant function on $S^2$. Then we have:
\begin{Teorema}\label{balancedness}
If $M$ is a scalar-flat K\"ahler surface, then, for  every smooth,  rotation invariant function $h$ on $S^2$, the $2$-form  $\Omega_h$  defines a  Hermitian metric on $Tw(M)$ and $S(M)$.  Moreover $d\Omega_h^2 =0$.
\end{Teorema}
\begin{proof}

As mentioned above, the  balanced condition $d\Omega_h^2=0$ reads 
\[
d\Omega_h^2=d(e^h \wedge \omega_{FS} \wedge \pi^*_p \omega_p) =0, 
\]
as $\omega^2_p$ is closed and $\omega_{FS}^2=0$ for dimensional reasons. From the considerations above $d(e^h \omega_{FS} \wedge \pi^*_p \omega_p) = de^h \wedge \omega_{FS} \wedge \pi^*_p\omega_p$.  But $d(e^h \wedge \omega_{FS} )= 0$ since $h$ is a function on $S^2$ and $\omega_{FS}$ is its volume form.  Since $S(M)$ is biholomorphic to $Tw(M)\setminus \{D_1,D_2\}$ we obtain that $\Omega_{h}$ defines also  a balanced metric on $S(M)$.
\end{proof}
It is not hard to study also the completeness of $\Omega_h$ on $S(M)$. We need the following general lemma.
\begin{Lemma}
        Let $\pi: (E,g_E)\to B$ be a Riemannian fiber bundle with fiber $F$. Assume that $g_E$ can be written as $g_E=\pi^* g_B+g_F$, where $g_B$ is a complete metric on $B$ and $g_F $ is a symmetric two form on $E$ which restricts to a complete metric on every fiber. Then $g_E$ is complete.
    \end{Lemma}
    \begin{proof}
        We will prove that any Cauchy sequence $(x_n)\subset E$ converges to a point $x\in E$. The sequence on the base $(\pi(x_n))$ is Cauchy as well, since for any smooth curve $\gamma: [0,1]\to E$ we have that 
        \[
        \int_0^1 \| \gamma'(t)\| \, dt \leq \int_0^1 \| (\pi\circ \gamma)'(t)\| \, dt
        \]
        By hypothesis, $(\pi(x_n))$ converges to a point $b\in B$. 
     For each $\pi(x_n)$, consider the (not necessarily unique) curve $\gamma_n$  satisfying  $l(\gamma_n)=dist_B(\pi(x_n),b)$. This curve exists by the completeness of $g_B$. One can consider their horizontal lifts $\tilde{\gamma}_n$, starting at $x_n$. By definition $\tilde{\gamma}_n$'s tangent space is purely horizontal, thus $l_E(\tilde{\gamma}_n)=l_B({\gamma}_n)$. Let $y_n$ be the endpoint of $\tilde{\gamma}_n$, which clearly sits over $b$. Consider the sequence of curves $\gamma_n \subset B$. Each of them is a geodesic starting at  $x_n$ and ending at $b$, thus $\| \gamma_n -\gamma_m \|_{C^2}$ is arbitrarily small for $n,m$ big enough. Since horizontal lifts are solutions of differential equations, $d(y_n, y_m)$ is arbitrarily small for $n,m$ big enough. By hypothesis, $y_n$ converges to a point $y\in \pi^{-1}(b)$. Then we have
        \[
        d(x_n, y) \leq d(x_n, y_n)+ d(y_n, y) \leq d_B(\pi(x_n), \pi(x_m)) + d(y_n, y)
        \]
        i.e. $x_n$ converges to $y$.

    \end{proof}
  
As a simple consequence of the  above lemma we have  the following result.
  
\begin{Corollary} \label{Cor-completness} 
If $M$ is  a scalar-flat K\"ahler surface and $h$ is a  smooth, rotation
invariant function  on $S^2$ such that   metric induced by $e^h \omega_{FS}$ on the  vertical part is complete,  then the the balanced metric on $S(M)$  defined by $\Omega_h$ is complete.
\end{Corollary}

\begin{rmk}  The completeness of the metric on $S(M)$   defined by $\Omega_h$ can be analyzed looking at the corresponding metric on $Tw(M)\setminus \{D_1,D_2\}$. The Fubini-Study metric on $\C\proj^1$, identified with $S^2$, is locally given by
    \[
    g_{FS} =d\phi^2+\cos \phi\, d\theta^2,
    \]
    where $\phi, \theta$ are, respectively, the latitude and longitude.  In these local coordinates, the metric
 corresponding to $e^h\omega_{FS}$ is given by 
    \[
  e^{h(\phi, \theta)} \, \left( d\phi^2+\cos \phi\, d\theta^2 \right)
    \]
    and in our case $h=h(\phi)$, since $h$ is   rotation
invariant.  Then the completeness of the metric induced by $\Omega_h$ can be checked  as follows. Consider  the curve $\theta=\theta_0$,  around the north pole,  the necessary condition is 
    \[
    \int_{0}^{\pi/2} e^{h(\phi)/2} \, d\phi = \infty.
    \]
    Since
    \[
    \int_{0}^{\pi/2} \sqrt{ e^{h(\phi)}(1+\cos^2(\phi) (\theta^{'}(\phi)) }  \, d\phi \geq
    \int_{0}^{\pi/2} e^{h(\phi)/2} \, d\phi 
    \]
    the condition is also sufficient. Analogous computations hold for the south pole. 
    This fiberwise condition is clearly  necessary for the completeness of $\Omega_h$, but it is also sufficient for the previous lemma.
    
\end{rmk}

From the proof of Theorem \ref{balancedness} we  can also  get the following:

\begin{Corollary}
Let $M$ be a scalar-flat K\"ahler surface with $b_-(M)=1$  and without non-parallel holomorphic vector fields. Then the forms $\Omega^2_{a,b}$, where $\Omega_{a,b} = a \, \pi^*\omega_p + b \, \omega_{FS}$ for $a,b>0$, define an open cone of closed positive $(2,2)$-forms in $H^{2,2}(Tw(M))$.
\end{Corollary}

For different values of $a$ and $b$ we  get different de Rham cohomology classes. This is implied by 
\[
\Omega_{a,b} ^2 \wedge \omega_{FS} = a^2 vol_{Tw(M)}, \quad \Omega_{a,b} ^2 \wedge \pi^*_p  \omega_p=  2 a b \, vol_{Tw(M)},
\]
and applying Stokes theorem.

\begin{Ex} Excluding complex surfaces with Ricci-flat K\"ahler metrics, compact scalar-flat K\"ahler surfaces have negative Kodaira dimension by a vanishing theorem of Yau \cite{Yau} and therefore  their minimal models are either ruled or biholomorphic to ${\mathbb {CP}}^2$, but by Chern-Weil theory  the latter possibility is excluded. 
The ruled surfaces are obtained as a quotient of  $H\times {\mathbb {CP}}^1$,  where $H$ is the upper half-plane.  They correspond to projectivized polistable rank $2$ bundles over a compact  Riemann surface \cite{Barth-Hulek-Peters-Van de Ven}; they admit
non-parallel holomorphic vector fields \cite{R-S}, but no parallel ones since their Euler characteristic is non-zero. The first examples which satisfy the conditions of the Corollary are the construction of  \cite{LeBrun} of scalar-flat K\"ahler metrics on blow-up of ruled surfaces. In  \cite{KLP} the authors extended the construction to scalar-fat metrics on ${\mathbb {CP}}^2\# 14\overline{\mathbb {CP}}^2$. In fact in \cite{KLP} they proved that  if one starts  with a compact K\"ahler surface $M$ whose total scalar curvature is non‑negative and does not admit a Ricci‑flat K\"ahler metric, then, after blowing up 
$M$  at sufficiently many points, the resulting complex surface admits a scalar‑flat K\"ahler metric. The results of \cite{LeBrun} and \cite{KLP} were generalized in \cite{R-S} where they proved the existence of scalar-flat K\"ahler metrics on ${\mathbb {CP}}^2\# 10\overline{\mathbb {CP}}^2$ for specific choice of the blown-up points, and for any other additional blow up. In particular,  a generic choice of further blow-ups produces examples for the Theorem above without holomorphic vector fields. 
\end{Ex}

\section{Modified hyperk\"{a}hler twistor space}\label{modified hyperk}
Let $Tw(M)$ be the twistor space of a compact simple hyperk\"ahler manifold  $(M,I,I,J, g)$,  where  simple means that 
$M$  is simply connected and the space of global holomorphic 2-forms on 
$M$ is one-dimensional. By definition, $Tw(M)$ has a holomorphic fibration $\pi: Tw(M)  \to  \C\proj^1$ and it is known (\cite{Verbitsky}) that for a dense and countable (in the analytic topology) subset $R\subset \C \proj^1$, $\pi^{-1}(\lambda)$ is projective if and only if $\lambda \in R$. Moreover if $(M,\lambda)$ is projective, clearly so is $(M,-\lambda)$. It is well known that $Tw(M)$ does not admit any K\"ahler metric, the following result shows that, as long as one deletes from $Tw(M)$ finetely many divisors, $Tw(M)$ remains non-K\"ahler. As a consequence, $S(M)$ is not Stein. This can be proven even more directly, since $S(M)$ is biholomorphic to $Tw(M)\setminus \{D_1, D_2\}$, where $D_1$ and $D_2$ are divisors, and $Tw(M)\setminus \{D_1, D_2\}$ is not Stein since it admits a holomorphic projection to a complex curve with compact complex submanifolds as fibers. The proof of the following Theorem is based on an idea of M. Verbitsky \cite{Verb2}.

\begin{Teorema} \label{Th4.1}  Let $Tw(M)$  be the twistor space of a compact simple hyperk\"ahler manifold
$(M,I,I,J,g)$.
    Removing a finite  number of divisors from $Tw(M)$  does not affect its non-K\"ahlerianity. In particular, $S(M)$ cannot be Stein.
\end{Teorema}
\begin{proof}
    Following \cite{Verbitsky}, 
let $(M, \lambda)$ be a projective hyperk\"ahler variety, which has  complex subvarieties of any codimension. Let $C$ be a odd dimensional such subvariety; for example, let $C$ be a curve. The antiholomorphic map $\sigma:  Tw(M) \to Tw(M)$ sends $(M, \lambda)$ to $(M, -\lambda)$ and $C$ is holomorphically identified with $C\subset (M, -\lambda) $ with reversed orientation. The fundamental classes $[C]$   and $[\sigma(C)]$ are opposite in integral homology, in particular their sum is zero. Since all the divisors in $Tw(M)$ are of the form $(M,\mu)$ for some $\mu \in \C \proj^1$, if one removes a finite number of them, an infinite  number of coupled projective fibers $(M,\lambda)$, $(M,-\lambda)$ will survive. This is a clear obstruction to the existence of a K\"ahler metric, since if a K\"ahler metric $\omega$ existed,   we would have  $$\int_{[\sigma(C)]}\omega=\int_{[C]} \omega=vol(C)$$ and $$\int_{[C]+[\sigma (C)]} \omega=0,
    $$
    which gives a contradiction.
    
\end{proof}
 \begin{rmk}
     The theorem is false in general for non-compact hyperk\"ahler manifolds. Consider for example the twistor space of $\mathbb{H}^n$ equipped with the standard flat hyperk\"ahler metric. Then $Tw(\mathbb{H}^n)$ can be identified with  $\text{Tot}(\mathcal{O}(1)^{2n})$. Removing one of the fibers of the twistor projection yields $\text{Tot}(\mathcal{O}(1)^{2n}\mid_{\C}) \cong  \C^{2n+1}  $.
 \end{rmk}

\begin{rmk} 
Note that $S(M)$ can be defined also  when $(M, I_1, I_2, I_3, g)$ is a  hypercomplex Hopf surface and we can prove that $(S(M), J)$  does not admit any  K\"ahler metric. The constructions of $S(M)$ and $J$ from  Section  2  also work when the Levi-Civita connection is replaced by the Weyl connection,  as  demonstrated by P. Gauduchon \cite{Gauduchon}.   In particular,  the K\"ahler scalar-flat condition could be replaced by  anti-self-dual  locally conformally K\"ahler one. By locally conformally K\"ahler structure  we mean a Hermitian structure   $(I, g, \omega)$ such that $d \omega = \theta \wedge \omega$,  with  $\theta$ a closed 1-form (called the Lee form).  Then,  the complex structure  $I$  is no longer parallel with respect to Levi-Civita, but  it is parallel with respect to the appropriate Weyl connection. When one considers the hypercomplex Hopf surface, then it is anti-self-dual and all complex structures in the family are Weyl-parallel.  We can fix one complex structure $I$ and again consider $S(M)$. It is again biholomorphic to $Tw(M)-\{D_1,D_2\}$,  where $D_i$ are the images of the sections corresponding to $\pm I$. However,  $S(M)$ will contain, as a complex subspace, a copy of a Hopf surface, which is non-K\"ahler. Consequently, $(S(M), J)$ does not admit  any K\"ahler metric.  

\end{rmk}

\end{document}